\documentclass[a4paper,12pt]{article}
\usepackage{fullpage}

\usepackage{cmbright}
\usepackage{microtype}
\usepackage[T1]{fontenc}
\usepackage{lmodern}
\usepackage[utf8]{inputenc}
\usepackage{bm} 
\usepackage{bbm}
\usepackage{amsfonts,amssymb,amsmath,amsthm,dsfont} 
\usepackage{xspace}
\usepackage{graphicx}
\usepackage{xcolor}
\usepackage{cite}
\usepackage{subfig}
\usepackage{comment}

\usepackage{hyperref}
\definecolor{darkgreen}{rgb}{0,0.4,0}
\definecolor{BrickRed}{rgb}{0.65,0.08,0}
\hypersetup{colorlinks=true,linkcolor=blue,citecolor=red,filecolor=BrickRed,urlcolor=darkgreen}

\usepackage[capitalize,nameinlink,noabbrev]{cleveref}

\usepackage{tikz}
\usepackage{tkz-tab}
\usepackage{booktabs}
\usepackage{multirow}
\usepackage{todonotes}
\usepackage{comment}

\usepackage{cleveref}

\newtheorem{theorem}{Theorem}
\newtheorem*{theorem*}{Theorem}
\newtheorem{lemma}[theorem]{Lemma}
\newtheorem{proposition}[theorem]{Proposition}

\newtheorem*{main-thm}{Main Theorem}
\crefname{assumption}{Assumption}{Assumptions}

\theoremstyle{definition}

\newtheorem{remark}[theorem]{Remark}

\newtheorem{conjecture}[theorem]{Conjecture}

\crefname{conjecture}{Conjecture}{Conjectures}

\newcommand{\eps}{\varepsilon}
\renewcommand{\to}{\rightarrow}

\newcommand{\whp}{with high probability}

\newcommand{\T}{\mathcal{T}}
\newcommand{\bT}{\mathbf{T}_{2n,g_n}}
\newcommand{\sep}{\text{sepsys}(\bT)}
\newcommand{\nto}{n \to \infty}
\newcommand{\fe}{fat eight}
\newcommand{\te}{thin eight}
\newcommand{\sncc}{SNCC}
\newcommand{\CC}{C(t)}
\newcommand{\inv}{^{-1}}
\newcommand{\Ts}{\T^{\text{simple}}}
\newcommand{\Tt}{\T^{\text{thin}}}
\newcommand{\Tf}{\T^{\text{fat}}}

\providecommand{\keywords}[1]
{
  \small	
  \noindent
  \textbf{{Keywords:}} #1
}

\begin{document}

\title{\textbf{The separating systole and the genus ratio of high genus triangulations}}

\author{%
Baptiste Louf\thanks{CNRS and Institut de Mathématiques de Bordeaux, France. Partially supported by ANR CartesEtPlus (ANR-23-CE48-0018) and ANR HighGG (ANR-24-CE40-2078-01).}}

\maketitle

\begin{abstract}
We show that the separating systole of high genus triangulations is of logarithmic order (in the size of the triangulation). Our methods also allow us to show an enumerative result, i.e. the convergence of the "genus ratio" for high genus triangulations. This complements the convergence of the "size ratio" that was proven in previous work with Budzinski.
\end{abstract}



\section{Introduction}
In this paper we study random maps on surfaces, in a regime in which both their size $n$ and their genus $g$ go to infinity. Recall that a \emph{map} is a discrete surface made by gluing polygons along their edges, and a \emph{triangulation} is a map built entirely from triangles.

In the planar case ($g=0$, $n \to \infty$), the geometry of random maps is now well understood. On a local scale, their structure is described by their local limit \cite{AS03,Kri05, CD06, Bud15}, while on a global scale, results cover their diameter and scaling limit~\cite{CS04,LG11, Mie11, Mar16}. These results rely heavily on enumerative formulas~\cite{Tut62,Tut63} and explicit bijections~\cite{CV81,Sch98these,BDG04}. Similar results have also been established for maps of fixed genus $g>0$ \cite{Chapuy:profile, BM22}.

When both $n$ and $g$ grow large, exact asymptotics are not available. However, for \emph{unicellular maps} (maps with a single face), a bijection~\cite{CFF13} provides a way to understand their geometry~\cite{ACCR13,Ray13a, Louf-expander, JansonLouf1, JansonLouf2}.

Beyond the unicellular case, a series of recent results study more general classes of maps, such as triangulations or maps with prescribed face degrees. Their local behavior has been uncovered~\cite{BL19,BL20}, confirming a conjecture of Benjamini and Curien (see ~\cite[Chapter 6]{B13book} and ~\cite{CurPSHIT}). Global properties have also been investigated: both the \emph{planarity radius}~\cite{Louf} (a discrete analogue of the injectivity radius) and the \emph{diameter}~\cite{BCL} grow logarithmically with the size of the map.

Results for high-genus maps often resemble those for another model of random surfaces: large-genus Weil--Petersson hyperbolic surfaces (including a strong conjectural link with unicellular maps~\cite{JansonLouf2, MP19}). Thanks again to precise asymptotics, many geometric properties have been established in that model (see e.g. \cite{Mir13,GPY11,NWX,MP19,PWX,AM,HMT}).

In this paper, we focus on a different global observable for random high-genus triangulations: the \emph{separating systole}, which is the length of the shortest separating but non-contractible cycle. Once again, it is analogous to existing results for hyperbolic surfaces~\cite{NWX,PWX}.

In all that follows, we fix $\theta\in(0,1/2)$ and a sequence $(g_n)$ such that $\frac{g_n}{n}\to \theta$. Let $\bT$ be a uniform triangulation of genus $g_n$ with $2n$ faces. 
\begin{theorem}\label{thm_sepsys}
With high probability\footnote{that is, with probability tending to $1$ as $\nto$, abbreviated later as whp.}, the separating systole of $\bT$ is of logarithmic order, i.e. there exists two constants $0<A<B$ (depending on $\theta$) such that whp
\[A\log(n)\leq \sep\leq B\log(n)\] 
\end{theorem}

Note that in this theorem, we do not require the shortest separating cycle to be simple. In fact, even the existence of a simple separating non-contractible cycle in triangulations is a longstanding conjecture \cite{MT01}[Chapter 5]. However, we do conjecture that \whp{}, it does exist in $\bT$, and what's more that it is again of logarithmic length. We will discuss this, and other conjectures such as the convergence of $\frac{\sep}{\log(n)}$ in probability, in more details in~\cref{sec_conj}.

The upper bound in~\cref{thm_sepsys} follows directly from a theorem of~\cite{CVHM14}. To establish the upper bound, we will resort to a similar approach as in \cite{Louf,BCL}: cut along cycles and estimate the number of triangulations with boundaries, and then  use a first moment method. However, contrary to previous works, we are dealing this time with non simple cycles, with can become quite intricate, especially as topology is involved. This is one of the main challenges of this work, and it is largely dealt with in~\cref{sec_geom}.

In the process of proving the lower bound, we can adapt some of the techniques to establish a new enumerative result on high genus triangulation that we call the convergence of the \emph{genus ratio}. Let us give more context. Let $\tau(n,g)$ be the number of triangulations of genus $g$ with $2n$ faces, we established in~\cite{BL19}, together with Budzinski, the convergence of the "\emph{size ratio}"
\begin{equation}\label{eq_ratio_size}
\frac{\tau(n-1,g_n)}{\tau(n,g_n)}\to \lambda(\theta),
\end{equation}
as $\nto$, where $\lambda$ is an explicit function.

It is a natural question to ask for a similar result when the genus varies and the size is fixed. We answer it with the following theorem:

\begin{theorem}\label{thm_ratio}
There exists an explicit function $\psi$ (see~\cref{eq_psi}) such that
\begin{equation} 
\frac{n^2\tau(n,g_n-1)}{\tau(n,g_n)}\to \psi(\theta),
\end{equation}
as $\nto$.
\end{theorem}

\paragraph{Acknowledgments:} Thanks to Thomas Budzinski, Guillaume Chapuy, Andrew Elvey-Price, Wenjie Fang and Michael Wallner for numerous joint works and enriching discussions about high genus geometry and enumeration. Thanks also to Arnaud de Mesmay and Hugo Parlier for interesting exchanges around \cref{thm_upper,conj_Barnette,conj_decomp}.

\section{Definitions and notations}
\subsection{Maps}
A \emph{map} $m$ is a finite graph (allowing loops and multiple edges) embedded on a compact, connected, oriented surface, considered up to homeomorphism. The connected components of the complement of the graph on the surface are called the \emph{faces} of $m$. Equivalently, a map can be constructed as a surface obtained by gluing a collection of polygons along their sides. The \emph{genus} of a map is the genus of its underlying surface.

We consider \emph{rooted} maps, that is, maps with a distinguished oriented edge called the \emph{root edge}. The face to the right of the root edge is the \emph{root face}, and the vertex at the origin of the root edge is the \emph{root vertex}.

The \emph{degree} of a face in a map is the number edge-sides that are incident to it. A \emph{triangulation} is a rooted map in which all faces have degree $3$. For integers $n \ge 1$ and $g \ge 0$, we denote by $\T(2n,g)$ the set of triangulations of genus $g$ with $2n$ faces, and $\tau(n,g)$ its cardinality. Finally, let $\mathbf{T}_{2n,g}$ a random uniform element of $\T(2n,g)$.

A \emph{triangulation with boundaries} is a map in which all faces have degree $3$ except for $k$ distinguished faces called the \emph{boundaries}. The boundaries must be simple (i.e., having the same number of vertices as edges), and disjoint in the sense that there is no vertex that is incident to two or more boundaries. Finally, for each boundary an edge incident to it is marked and called a root. We denote by $\T_{p_1,…,p_k}(m,g)$ the set of triangulations of genus $g$ with $m$ faces and $k$ boundaries, with the $i$-th boundary of size $p_i$.
\subsection{Paths and cycles}
In a map $m$, a \emph{path} $p$ of length $\ell$ is the data of a finite sequence of oriented edges $(\vec{e}_i)_{1 \leq i \leq \ell}$ such that for each $1\leq i<\ell$, the endpoint of $\vec{e}_i$ is the starting point of $\vec{e}_{i+1}$. Given a path $p=(\vec{e}_i)_{1 \leq i \leq \ell}$, we call $v_i$ the endpoint of $\vec{e}_i$, and $v_0$ the starting point of $\vec{e}_1$. We will also say that $v_0$ is the starting point of $p$, and $v_\ell$, its endpoint.

A path $c$ is called a \emph{cycle} if $v_0=v_\ell$. A path $p$ is \emph{simple} if all its $v_i$'s are distinct, and a cycle $c$ is simple if all its $v_i$'s are distinct expect for $v_0=v_\ell$.

Given two paths $p$ and $q$ such that the endpoint of $p$ is the starting point of $q$, we write $pq$ for the concatenated path. We also write $p\inv$ for the path obtained by "reversing" $p$. (Note that this is only a "homotopic inverse", i.e. if $p$ is non-empty, $pp\inv$ is not an empty path!)

\subsection{Homotopy and homology}

Any path (resp. cycle) of $t$ can be seen as an oriented (resp. closed) curve on a surface $S$, and hence we can study their \emph{homotopy} and \emph{homology}. Let us recall some basic facts, we refer to~\cite{Gib} for more details. Two curves on a $S$ are said to be \emph{homotopic} if one can be continuously deformed into the other. A curve is said to be \emph{contractible} (or null-homotopic) if it is homotopic to a single point. Homotopy classes of closed curves on a surface form a group (by concatenation), and the abelianization of this group is called the (first) homology group of the surface. This defines an equivalence relation of homology between closed curves. A curve is said to be \emph{separating} (or null-homologous) if it is homologous to a contractible curve. The simplest example of a non-contractible but separating curve is a simple curve that, when cut, separates the surface into two subsurfaces both of positive genus, but non simple curves provide many more examples\footnote{for instance, if $\gamma$ and $\mu$ are two non homotopic non separating curves that intersect transversely at a single point, then $\gamma\mu\gamma\inv\mu\inv$ provides such an example}. 

Finally, the \emph{separating systole} of a triangulation $t$ is the length of its shortest non-contractible but separating cycle (abbreviated later as "\sncc{}"). We will write $\sep$ for the separating systole of $\bT$.

\section{Preliminaries: enumeration}\label{sec_enum}
In this section, we present all the enumerative tools that we will need for the proof of the main theorem. Several of them already exist and we recall them, but we also need to prove some more specific results. As a byproduct, we obtain \cref{thm_ratio}.

\subsection{Previous results}

We start with an exact recurrence formula for triangulations, the \emph{Goulden--Jackson formula}~\cite{GJ08}.

\begin{theorem}[Goulden--Jackson formula]\label{GJ}
\begin{multline*}
(n+1)\tau(n,g)= 4n(3n-2)(3n-4)\tau(n-2,g-1)+4(3n-1)\tau(n-1,g)\\
+4\sum_{\substack{i+j=n-2\\i,j\geq 0}}\sum_{\substack{g_1+g_2=g\\g_1,g_2\geq 0}}  (3i+2)(3j+2)\tau(i,g_1)\tau(j,g_2)+2\mathbbm{1}_{n=g=1}.
\end{multline*}
\end{theorem}

Then, we have a non asymptotic inequality between triangulations with and without boundaries. The proof is elementary, and can be found for instance in~\cite{BCL}.

\begin{lemma}\label{lem_filling_boundaries}
Let $k,m \geq 1$ and $g \geq 0$. For any $k$-uple $p_1,…,p_k$, writing $p=p_1+…+p_k$, we  have the inequality
\[|\mathcal{T}_{p_1,…,p_k}(m,g)|\leq (3m+3p)^{k-1}\tau\left(\frac{m+p}2,g\right).\]
\end{lemma}

We now turn to enumeration results that are more specific to the high genus regime.
First, we have some exact results
\begin{theorem}
As $\nto$, we have
\begin{equation}\label{eq_ratioo}
\frac{\tau(n-1,g)}{\tau(n,g)}=\lambda(g/n)+o(1)
\end{equation}
and
\begin{equation}\label{eq_asymptoo}
\tau(n,g)=n^{2g}\exp(nf(g/n)+o(n))
\end{equation}
where $\lambda$ is a strictly decreasing function with $\lambda(1/2)=0$, $f$ is a concave function and the small $o$'s are uniform for $g\in\left[0,\frac{n+1}{2}\right]$.
\end{theorem}
This was proved in~\cite{BL19}, except for the concavity of $f$ which was proven in \cite{BCL}. In that same paper, a general asymptotic high genus inequality was proven.

\begin{proposition}\label{prop_asympto_ratio}
For all $\theta \in \left( 0,\frac{1}{2} \right)$, there is a constant $a_{\theta} \in (0,1)$ such that the following holds. Let $(g_n)$ be a sequence such that $0 \leq g_n \leq \frac{n+1}{2}$ for every $n$ and $\frac{g_n}{n} \to \theta$. For all integers $m$ and $h$ satisfying $0\leq m\leq \frac{n}{2}$ and $0\leq h\leq \min \left( \frac{g_n}{2}, \frac{m+1}{2} \right)$, we have
\[
\frac{\tau(n,g_n)}{\tau(n-m,g_n-h)}\geq a_{\theta}^{h}\frac{n^{2g_n}}{(n-m)^{2(g_n-h)}}\exp \left( f\left(\frac{g_n}{n}\right) n -f\left(\frac{g_n-h}{n-m}\right)(n-m)+ o(m) \right),
\]
where the $o(m)$ is uniform in $(m,h)$ as $n \to +\infty$ (that is, it is bounded by $m\eps(n)$ with $\eps(n) \to 0$ as $n \to +\infty$).
\end{proposition}

\subsection{Additional results}

To these existing results, let us add a couple more that will be of use in this paper.
First, as a direct consequence of \cref{eq_ratioo,GJ}, we have the following:
\begin{lemma}\label{lem_asympto_genus_ratio}
Let $h$ be a fixed integer, $C$ a constant and $|\ell|\leq C\log n$. Then 
\[\frac{\tau(n+\ell,g_n-h)}{\tau(n,g_n)}\leq n^{-2h+o(1)} (\lambda(\theta)+o(1))^{-\ell}\]
where the $o(1)$ is uniform for all $-C\log n\leq \ell\leq C\log n$.
\end{lemma}

The next lemma  will allow us to (roughly) enumerate pairs of triangulations.

\begin{lemma}\label{lem_asympto_cut}
Let $a,b$ be two non negative integers, and $(g_n)$ as in the previous lemma.
Then
\[\sum_{n_1+n_2=n}\sum_{h_1+h_2=g_n\atop h_1,h_2\geq 1}n_1^a\tau(n_1,h_1)n_2^b\tau(n_2,h_2)\leq n^{\max(a,b)-2+o(1)}\tau(n,g_n)\]
\end{lemma}

\begin{proof}
Assume wlog $a\geq b$, we will show
\begin{equation}\label{eq_ineq_smaller_n2}
\sum_{n_1+n_2=n\atop n_1\geq n_2}\sum_{h_1+h_2=g_n\atop h_1,h_2\geq 1}n_1^a\tau(n_1,h_1)n_2^b\tau(n_2,h_2)\leq n^{a-2+o(1)}\tau(n,g_n)
\end{equation}
which directly implies the desired result.

We will split the analysis of the RHS of this inequality in four cases. Let $\eps=\min(\sqrt{a_\theta},\theta)$, where $a_\theta$ is the constant of \cref{prop_asympto_ratio}.

\paragraph{\textbf{Case 1:}} We consider the case $n_2\geq \eps n$. Then, by \cref{eq_asymptoo}, and in particular by concavity of $f$, we have 
\begin{align*}
\frac{n_1^a\tau(n_1,h_1)n_2^b\tau(n_2,h_2)}{\tau(n,g_n)}&\leq \frac{n_1^{2h_1+a}\exp(n_1f(h_1/n_1))n_2^{2h_2+b}\exp(n_2f(h_2/n_2))}{n^{2g_n}\exp(nf(g_n/n))}e^{o(n)}\\
&\leq \frac{n_1^{2h_1+a}n_2^{2h_2+b}}{n^{2g_n}}e^{o(n)}.
\end{align*}
But $n_1,n_2\leq (1-\eps)n$, therefore
\begin{equation}\label{eq_large_n2}
\frac{n_1^a\tau(n_1,h_1)n_2^b\tau(n_2,h_2)}{\tau(n,g_n)}\leq (1-\eps)^{2g_n}e^{o(n)}.
\end{equation}
In other words, terms where $n_2$ is too large are exponentially small.

\paragraph{}Now, we tackle the two other cases. Using \cref{prop_asympto_ratio}, once again with~\cref{eq_asymptoo}, we have that
\begin{equation}\label{eq_ratio_base}
\frac{n_1^a\tau(n_1,h_1)n_2^b\tau(n_2,h_2)}{\tau(n,g_n)}\leq a_\theta^{-h_2}\frac{n_1^{2h_1+a}n_2^{2h_2+b}}{n^{2g_n}}e^{o(n_2)}.
\end{equation}

\paragraph{\textbf{Case 2:}} Consider now the case $\frac{5+b}{\theta}\log n\leq n_2\leq \eps n$.
First, by definition of $\eps$, we have 
\[a_\theta^{-h_2}\frac{n_2^{2h_2}}{n^{2h_2}}\leq a_\theta^{-h_2}(\sqrt{a_\theta})^{2h_2}= 1.\]

Note that $h_2\leq \frac{n_2+1}{2}\leq (\theta/2+o(1)) n$ by definition of $\eps$, hence $h_1\geq (1+o(1))\frac{\theta}{2} n$ and thus
\[\frac{n_1^{2h_1}}{n^{2h_1}}=\frac{(n-n_2)^{2h_1}}{n^{2h_1}}=\exp(-(1+o(1)2h_1\frac{n_2}{n})=\exp(-\theta n_2+o(n_2))\leq n^{-5-b+o(1)}\]
where in the last inequality we used the fact that $n_2\geq \frac{5+b}{\theta}\log n$. Using  \cref{eq_ratio_base} and the two inequalities above
\begin{equation}\label{eq_medium_n2}
\frac{n_1^a\tau(n_1,h_1)n_2^b\tau(n_2,h_2)}{\tau(n,g_n)}\leq \frac{n_1^{2h_1+a}n_2^{2h_2+b}}{n^{2g}}e^{o(n_2)}\leq n^{a+b} n^{-5-b+o(1)}\leq n^{-5+a+o(1)}
\end{equation}

\paragraph{\textbf{Case 3:}} We're almost done, take now $n_2\leq \frac{5+b}{\theta}\log n$ and $h_2\geq 2$.
By \cref{eq_ratio_base}, we have
\begin{equation}\label{eq_small_n2}
\frac{n_1^a\tau(n_1,h_1)n_2^b\tau(n_2,h_2)}{\tau(n,g_n)}\leq n^{a+o(1)} \left(\frac{(5+b)\log n}{a_\theta n}\right)^{2h_2}\leq n^{a+o(1)}\left(\frac{(5+b)\log n}{a_\theta n}\right)^{4}=n^{a-4+o(1)}
\end{equation}

\paragraph{}Note that in the RHS of\cref{eq_ineq_smaller_n2}, there is less than $n^2$ terms. Therefore, combining \cref{eq_large_n2,eq_medium_n2,eq_small_n2}, we obtain that
\[\sum_{n_1+n_2=n\atop n_1\geq n_2}\sum_{h_1+h_2=g_n\atop h_1\geq 1,h_2\geq 2}n_1^a\tau(n_1,h_1)n_2^b\tau(n_2,h_2)= \sum_{n_1+n_2=n\atop n_2\leq (5+b)\log n}n_1^a\tau(n_1,g_n-1)n_2^b\tau(n_2,1)+o(n^{a-2+o(1)}\tau(n,g_n)).\]

\paragraph{\textbf{Case 4:}} To show~\cref{eq_ineq_smaller_n2} (and thus finish the proof of the lemma), it remains to show that 
\[\sum_{n_1+n_2=n\atop n_2\leq (5+b)\log n}n_1^a\tau(n_1,g_n-1)n_2^b\tau(n_2,1)\leq n^{a-2+o(1)}\tau(n,g_n)).\]
But, on the one hand, by \cref{lem_asympto_genus_ratio}, uniformly for all $n_2\leq (5+b)\log n$, we have
\[\frac{\tau(n-n_2,g_n-1)}{\tau(n,g_n-1)}\leq  n^{-2+o(1)}\lambda(\theta)^{n_2}.\]
And on the other hand, it is well-known~(see for instance~\cite{Kri07}) that the generating series
\[H_b(x)=\sum{n\geq 0} n^b\tau(n,1)x^n\]
has radius of convergence $\lambda(0)>\lambda(\theta)$, therefore $H_b(\lambda(\theta))<\infty$ and
\[\sum_{n_1+n_2=n\atop n_2\leq (5+b)\log n}n_1^a\tau(n_1,g_n-1)n_2^b\tau(n_2,1)\leq n^{a-2+o(1)}\sum_{k\geq 0}\lambda(\theta)^kk^b\tau(k,1)=n^{a-2+o(1)}H_b(\lambda(\theta)).\]
This finishes the proof.
\end{proof}

\subsection{The genus ratio: proof of~\cref{thm_ratio}}

Thanks to the results of the previous subsection, we have (almost) everything we need to prove~\cref{thm_ratio}. The only additional ingredient we need is similar to~\cref{lem_asympto_cut}, which roughly says that in the double sum of the Goulden--Jackson recurrence formula, only the terms where $g_1=0$ or $g_2=0$ contribute. The following lemma shows that, on top of that, only the terms where $n_1$ or $n_2$ is very small contribute. More precisely:

\begin{lemma}\label{lem_asympto_cut_2} As $\nto$ with $\frac{g_n}{n}\to\theta\in(0,1/2)$,
\[\sum_{n_1+n_2=n\atop n_2\geq \log(n)}(3n_1+2)\tau(n_1,g_n)(3n_2+2)\tau(n_2,0)=o(\tau(n,g_n))\]
\end{lemma}

\begin{proof}
On the one hand, since $\lambda$ is a decreasing function, and by~\cref{lem_asympto_genus_ratio}, we have for all $n_2$
\[\frac{\tau(n-n_2,g_n)}{\tau(n,g_n)}\leq (\lambda(g_n/n)+o(1))^{n_2}\]
with uniform $o(1)$.

On the other hand, it is classical~\cite{Tut62} that
\[\tau(n,0)\sim C n^{-5/2} \lambda(0)^{-n}\]
for some constant $C$ as $\nto$, therefore for all $n$ there exists a constant $c$ such that
\[(3n+2)\tau(n,0)\leq c\lambda(0)^{-n}\]
Therefore
\begin{align*}
\frac{1}{n\tau(n,g)}\sum_{n_1+n_2=n\atop n_2\geq \log(n)}(3n_1+2)\tau(n_1,g_n)(3n_2+2)\tau(n_2,0)&\leq \sum_{n_1+n_2=n\atop n_2\geq \log(n)} (3+o(1))(\lambda(g_n/n)+o(1))^{n_2}c\lambda(0)^{-n_2}\\
&=O\left(\left(\frac{\lambda(\theta)+o(1)}{\lambda(0)}\right)^{\log n}\right)=o(1)
\end{align*}
because $\lambda(\theta)<\lambda(0)$.
\end{proof}

We come back to the proof of~\cref{thm_ratio}.
Applying~\cref{lem_asympto_cut,lem_asympto_cut_2} to the Goulden--Jackson recursion~\cref{GJ}, one obtains
\[(1+o(1))\tau(n,g_n))=36n^2\tau(n-2,g_n-1)+12\tau(n-1,g_n)+24\sum_{k=0}^{\log(n)}\tau(n-2-k,g_n)(3k+2)\tau(k,0)\]
By~\cref{eq_ratioo}, this entails
\begin{equation}\label{eq_sum_truncated}
(1+o(1))=36\lambda(\theta)^2n^2\frac{\tau(n,g_n-1)}{\tau(n,g)}+12\lambda(\theta)+24\lambda(\theta)^2\sum_{k=0}^{\log(n)}(\lambda(\theta)+o(1))^k(3k+2)\tau(k,0)
\end{equation}
Let $H(x):=\sum_{k\geq 0}x^k(3k+2)\tau(k,0)$. It is well known~\cite{Tut62} that $H$ has a singularity at $x=\lambda(0)$ and is finite at its singularity. 
Bu uniformity of the small $o$'s, we then have
\[24\sum_{k=0}^{\log(n)}(\lambda(\theta)+o(1))^k(3k+2)\tau(k,0)=H(\lambda(\theta)+o(1))+o(1))=H(\lambda(\theta))+o(1)).\]
Which together with \cref{eq_sum_truncated} yields
\[n^2\frac{\tau(n-2,g_n-1)}{\tau(n,g)}\to \frac{1}{36\lambda(\theta)^2}\left(1-12\lambda(\theta)-24\lambda(\theta)^2H(\lambda(\theta))\right)\]
as $\nto$.

Therefore, by~\cref{eq_ratio_size}, \cref{thm_ratio} holds with 
\begin{equation}\label{eq_psi}
\psi(\theta)=\frac{1}{36\lambda(\theta)^4}\left(1-12\lambda(\theta)-24\lambda(\theta)^2H(\lambda(\theta))\right).
\end{equation}

\section{Preliminaries: geometry}\label{sec_geom}

In a triangulation $t$, we call \emph{\te} any tuple $(c_1,c_2,p)$ such that $c_1$ and $c_2$ are two simple non-separating cycles which are not homotopic and do not intersect and $p$ is a simple path whose starting point lies on $c_1$, whose endpoint lies on $c_2$ but does not share any other vertex (and thus any edge) with $c_1$ and $c_2$.

Now, let us define the notion of \emph{\fe}. Take a tuple $(c,p)$ such that $c$ is a simple non-separating cycle and $p$ is a simple path whose starting point is the starting point of $c$ whose  endpoint lies on $c$, but does not share any other vertex (and thus any edge) with $c$. In that case, we can write $c=qr^{-1}$, where $q$ and $r$ are two simple paths that share starting points and end points with $p$. Then $pr^{-1}$ and $qp^{-1}$ are two simple cycles, and we say that $(c,p)$ is a \fe{} if neither of these cycles is contractible.

Actually, we will add one special case to this definition: if $c_1$ and $c_2$ are two simple non separating cycles that intersect at a single vertex but do not share any other vertex or edge, we say that $c_1,c_2$ is a \fe{} if they intersect transversely (i.e., going clockwise around the intersection vertex, edges of $c_1$ and $c_2$ alternate), and that it is a \te{}  otherwise (i.e. if they intersect tangentially).

\begin{figure}
\centering
\includegraphics[scale=0.5]{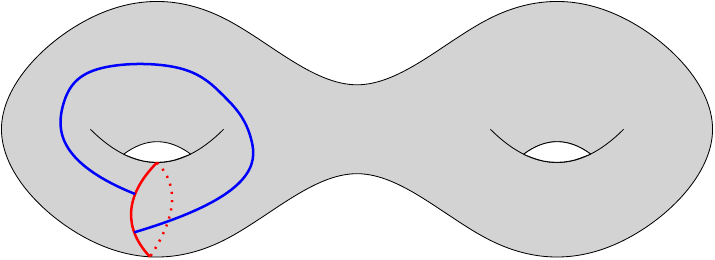}
\includegraphics[scale=0.5]{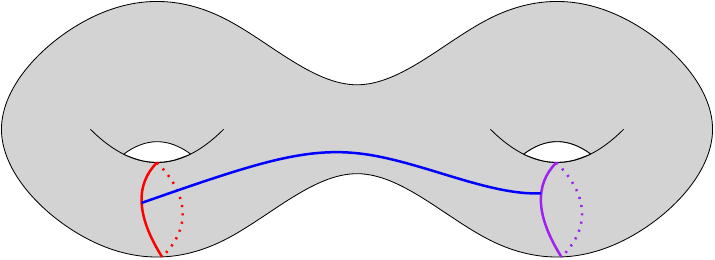}
\caption{Left: a \fe{}, right: a \te{}}\label{fig_thin_fat}
\end{figure}

The goal of this section is to prove the following:
\begin{theorem}\label{thm_eights}
If $t$ is a triangulation whose separating systole is $L$, then $t$ contains either:
\begin{itemize}
\item a simple non-contractible separating cycle of length $\leq L$;
\item a \te{} of total length $\leq L$;
\item a \fe{} of total length $\leq L$.
\end{itemize}
\end{theorem}

With this theorem in hand, it will be much easier to prove the lower bound of~\cref{thm_sepsys}: it will suffice to forbid short occurrences of \te{s}, \fe{s} and simple non-contractible separating cycles in $\bT$, instead of dealing with the infinite family of possible non-contractible separating cycles.

We need a few more definitions and technical lemmas before proving~\cref{thm_eights}. 
A cycle $C$ will be called \emph{reduced} if it can be written $C=\tilde\gamma C'$ where $\tilde\gamma$ is a simple cycle, and $C'$ a (possibly empty) cycle.

\begin{lemma}\label{lem_reduced}
Up to changing its starting point, any cycle $C$ is reduced. In that case, $C$ can be written in reduced form, that is, one can decompose $\tilde\gamma$ as $\tilde\gamma=pr^{-1}$, where $p$ and $r$ are two simple paths whose starting and endpoints coincide, but do not share any other vertices or edges (with $p$ possibly empty). Then, setting $\gamma=r^{-1}p$, one can write  $C=p\gamma^kP'$ where $k$ is a positive integer, and $P'$ is a path that starts with an edge that does not belong to $\gamma$ (or the empty path).
\end{lemma}

\begin{proof}
If $C$ is the power of a simple cycle, the result is obvious, now assume this is not the case. Let $(\vec{e}_i)_{1 \leq i \leq \ell}$ be its edges and $(v_i)_{0 \leq i \leq \ell}$ its vertices. Since $C$ is non simple, there exists $0\leq a<b<\ell$ such that $v_a=v_b$. Taking $b$ smallest possible, up to changing the starting point of $C$, one can assume that $a=0$. Set $\tilde\gamma=(\vec{e}_i)_{1 \leq i \leq b}$, the cycle $C$ is now in reduced form. Since $C$ is not the power of a simple cycle, there must exist a smallest $c>b$ such that $\vec{e}_c\neq\vec{e}_i$ for all $1\leq i\leq b$. Then one can set  $p=(\vec{e}_i)_{1 \leq i \leq (c-1\mod b)}$ (note that $p$ is empty if $c=1\mod b$), $r=\left((\vec{e}_i)_{(c\mod b) \leq i \leq b}\right)\inv$, and $P'=(\vec{e}_i)_{c \leq i \leq \ell}$.
\end{proof}

\subsection{A nice cycle}

In order to prove~\cref{thm_eights}, we will exhibit a particular nice cycle that contains either a \te{}, a \fe{}, or is simple. This nice cycle will be a minimum with respect to a well chosen partial order. Let us define it.

For any cycle $C=(\vec{e}_i)_{1 \leq i \leq \ell}$, we define its \emph{set-length} $SL(C)=(SL_i(C))_{1 \leq i \leq \ell}$ where $SL_i(C)$ is the cardinality of the set $\{\vec{e}_i|1 \leq j \leq i\}$. 

This defines a weakly increasing sequence with increments $0$ or $1$. We can then define a partial order on cycles using the lexicographic order on set-length. For each pair $C,C'$ of cycles of same length $\ell$ such that $SL(C)\neq SL(C')$, there exists $j=\min\{1\leq i\leq\ell|SL_i(C)\neq SL_i(C')\}$. We say that $C\succ C'$ if and only if $SL_j(C)> SL_j(C')$.

We are now ready to define our nice cycle. In the rest of the section, we fix a triangulation $t$ with separating systole equal to $L$. Consider the set of reduced \sncc{s} of $t$ with length $L$ (which is non-empty thanks to~\cref{lem_reduced}), and in this set, pick a cycle $\CC$ which is minimal with respect to the order $\succ$ defined above.

Let us write $\CC$ in reduced form
\[\CC=p\gamma^k P'.\]
Now we will prove many nice but technical properties about $\CC$. The first lemma is straightforward:

\begin{lemma}
The simple cycle $\gamma$ is non separating.
\end{lemma}

The following lemma will be useful to search for a \te{}.
\begin{lemma}\label{lem_te}
The path $P'$ cannot be written as $P'=q\eta P''$ where 
\begin{itemize}
\item $P''$ is a path;
\item $\eta$ is a simple cycle homotopic to $\gamma$ but disjoint from $\gamma$
\item $q$ is a simple path that has its starting point on $\gamma$, its endpoint on $\eta$, but does not share any other vertex or edge with $\gamma$ and $\eta$.
\end{itemize} 
\end{lemma}
See \cref{fig_lem_te} for an illustration.
\begin{figure}
\center
\includegraphics[scale=0.8]{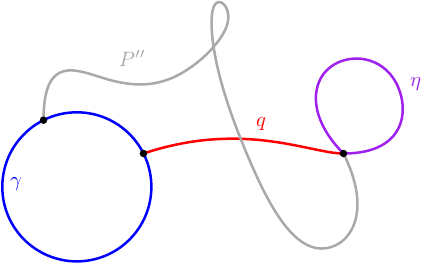}
\caption{The paths and cycles introduced in \cref{lem_te}}\label{fig_lem_te}
\end{figure}
\begin{proof}
Assume by contradiction that it is the case, i.e. $\CC=p\gamma^kq\eta P''$. We will show that it violates minimality in one way or another.

First, observe that $q \eta q\inv$ is homotopic to $\gamma$. Therefore, the two cycles $\CC'=p\gamma^{k+1}q P''$ and $\CC''=pq\eta^{k+1} P''$ are both homotopic to $\CC$. Since $\CC$ has minimal length among \sncc{s}, this implies that $\gamma$ and $\eta$ must have the same length. Therefore $\CC'$ is also a reduced \sncc{} of minimal length. But there is a contradiction, since $\CC\succ \CC'$.
\end{proof}

If instead, we search for a \fe{}, we will use the next lemma.
\begin{lemma}\label{lem_fe}
One cannot find all the following objects (as illustrated in \cref{fig_lem_fe}):
\begin{itemize}
\item two simple paths $p_1$ and $p_2$ such that $\gamma=p_1p_2$;
\item a simple path $q$ that shares its starting point and endpoint with $p_1$ but does not share any other vertex or edge with $\gamma$, such that $qp_1\inv$ is contractible;
\item a path $P''$.
\end{itemize}
such that $P'=qP''$.
 
\end{lemma}

\begin{figure}
\center
\includegraphics[scale=0.8]{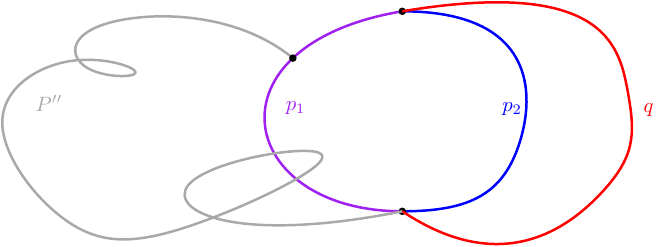}
\caption{The paths introduced in \cref{lem_fe}}\label{fig_lem_fe}
\end{figure}

\begin{proof}
As before, assume it is the case by contradiction, i.e. $\CC=p\gamma^kqP''$, and let us violate minimality.

Consider $\CC'= p\gamma^{k}p_1 P''$, it is clearly homotopic to $\CC$, and what's more, $\CC\succ \CC'$. Therefore $q$ must be strictly shorter than $p_2$ in order not to violate minimality.
But if this is the case, then $\gamma'=qp_2$ is homotopic to and strictly shorter than $\gamma$, therefore $\CC''=p\gamma'^{k}P'$ is a \sncc{} that is strictly shorter than $\CC$, a contradiction.
\end{proof}

\begin{remark}
Since $\CC=p\gamma^{k}P'$ is of minimal length, it is natural to think that $k$ should be equal to $1$, see~\cref{conj_decomp}.
\end{remark}

\subsection{Proof of~\cref{thm_eights}}

We are now ready to prove~\cref{thm_eights}.

Assume $\CC$ is not simple, otherwise we are done. In that case, $P'$ is non empty.
Write \[\CC=p\gamma^{k}P'=(\vec{e}_i)_{1 \leq i \leq L}.\]
As before, write $v_0$ for the starting point of $\CC$ and $v_i$ for the endpoint of $\vec{e}_i$. 
Let also $m>0$ be the length of $\gamma$ and $0\leq r<m$ the length of $p$.

Now, let $b>km+r$ be minimal such that there exists $a<b$ with $v_a=v_b$.
In other words, $P'$ starts with a path that will eventually hit $\gamma$ or itself. There are two cases:

If $a> km+r$ (the path hits itself but not where it started), set $q=(\vec{e}_i)_{km+r+1 \leq i \leq a}$, it is a simple path, and $\eta=(\vec{e}_i)_{a+1 \leq i \leq b}$, it is a simple cycle. If $\eta$ was contractible or separating, that would contradict the minimality of $\CC$, hence $\eta$ is non-separating. And by~\cref{lem_te}, it cannot be homotopic to $\gamma$. Therefore $(\gamma,q,\eta)$ is a \te{} of $t$.

If $a< km+r$ and $a\neq r \mod m$ (the path hits $\gamma$ but not where it left it), let $q=(\vec{e}_i)_{km+1 \leq i \leq b}$, it is a simple path, and we want to show that $(\gamma, q)$ forms a \fe{}. We can write $\gamma= p_1p_2$ where $p_1$, $p_2\inv$ and $q$ share staring points and endpoints but no other edge. By \cref{lem_fe}, $qp_1\inv$ is not contractible. It remains to show that $p_2q$ is not contractible. But if this was the case, we would have $\CC=p\gamma^{k-1} p_1p_2qP''$, and clearly the cycle $p\gamma^{k-1} p_1P''$ would be homotopic to and strictly shorter than $\CC$, a contradiction. We have thus shown that $(\gamma, q)$ is a \fe{}.

Finally, if $a= km+r$, this is the case where the path returns hits itself and $\gamma$ at the same time, and we have a special case which either gives us a \te{} or a \fe{}.

\section{Proof of~\cref{thm_sepsys}}

Let us show now that~\cref{thm_sepsys} holds. For the upper bound, we can directly apply the following deterministic result of \cite{CVHM14}:

\begin{theorem}[Corollary 3.2, part 3 in \cite{CVHM14}]\label{thm_upper}
There exists a constant $c''>0$ such that every triangulation of $\T(2n,g)$ contains a separating cycle of length at most $c''\sqrt{2n/g}\log (g)$.
\end{theorem}

Therefore, the upper bound in~\cref{thm_sepsys} follows with any $B>c''\sqrt{2/\theta}$.

Now we want to prove the lower bound. By~\Cref{thm_eights}, it suffices to show that, whp, $\bT$ doesn't contain any simple separating cycle, \te{} of \fe{} of length less than $A\log(n)$ for a well chosen $A$. Using the first moment method, we reduce this problem to a counting problem: proving \cref{prop_simple,prop_fat,prop_thin}.

Let $\Ts(2n,g,\ell)$ (resp. $\Tt(2n,g,\ell)$, $\Tf(2n,g,\ell)$) be the number of triangulations of $\T(2n,g)$ with a marked simple \sncc{} (resp. \te{}, \fe{}) of total length $\ell$. 

The method now consists in constructing injective operations to bound the cardinality of these sets. Roughly speaking, we will cut along cycles (and paths) to obtain one or two triangulations with boundaries, and then we will apply the bounds of \cref{sec_enum}.

\subsection{Counting simple cycles}
In this section, we prove that $\bT$ does not contain small simple \sncc{s}, that is: 

\begin{proposition}\label{prop_simple}
Let $A=\frac{1}{2\log\lambda(\theta)}$, we have
\[\sum_{\ell\leq A\log n} \frac{|\Ts(2n,g_n,\ell)|}{\tau(n,g_n)}\to 0\]
as $\nto$.
\end{proposition}
Let us start with an inequality that holds for all $n,g,\ell$.

\begin{lemma}\label{lem_simple}
We have 
\[|\Ts(2n,g,\ell)|\leq 6n\sum_{n_1+n_2=n+\ell}\sum_{g_1+g_2=g\atop g_1,g_2\geq 1}\tau(n_1,g_1)\tau(n_2,g_2)\]
\end{lemma}
\begin{proof}
We prove the lemma by an injective operation. 

Let $t,c$ be an element of $\Ts(2n,g,\ell)$ where $t$ is the triangulation and $c$ the marked cycle.
Pick an arbitrary edge on the $c$ and declare it is the root, while the original root becomes a marked oriented edge ($6n$ choices). Cut along $c$ to obtain a pair of maps in $\T_\ell(m_1,g_1)\times\T_\ell(m_2,g_2)$ where $m_1+m_2=n$ with $m_1\equiv m_2\equiv \ell \mod 2$ and $g_1+g_2=g$ with $g_1,g_2\geq 1$ . To go back, one simply glues the two boundaries such that the root edges match. Hence we have an injective operation and we have shown that
\[|\Ts(2n,g,\ell)|\leq 6n\sum_{m_1+m_2=2n\atop m_1\equiv m_2\equiv \ell \mod 2}\sum_{g_1+g_2=g\atop g_1,g_2\geq 1}|\T_\ell(m_1,g_1)\times\T_\ell(m_2,g_2)|.\]
Invoking~\cref{lem_filling_boundaries} finishes the proof.
\end{proof}

We can now prove ~\cref{prop_simple}:
\begin{proof}[Proof of~\cref{prop_simple}]
Take $\ell\leq A\log n$, then by \cref{lem_simple,lem_asympto_genus_ratio}, we get
\begin{align*}
\Ts(2n,g_n,\ell)&\leq n^{1+o(1)}\sum_{n_1+n_2=n+\ell}\sum_{g_1+g_2=g\atop g_1,g_2\geq 1}\tau(n_1,g_1)\tau(n_2,g_2) \\
&\leq n^{-1+o(1)}\tau(n+\ell,g_n)= n^{-1+o(1)}\lambda(\theta)^{\ell+o(\ell)}\tau(n,g_n) \leq n^{-1/2+o(1)}\tau(n,g_n).
\end{align*}

where in the last inequality we use the fact that $\lambda(\theta)^{A\log(n)}\leq \sqrt{n}$.
Summing this over all $\ell \leq A\log n$ yields the desired result.
\end{proof}
\subsection{Counting \fe{s}}

In this section, we prove that $\bT$ does not contain small \fe{s}, that is:  

\begin{proposition}\label{prop_fat}
Let $A=\frac{1}{2\log\lambda(\theta)}$, we have
\[\sum_{\ell\leq A\log n} \frac{\Tf(2n,g_n,\ell)}{\tau(n,g_n)}\to 0\]
as $\nto$.
\end{proposition}

Let us start with an inequality that holds for all $n,g,\ell$.

\begin{lemma}\label{lem_fat}
We have 
\[|\Tf(2n,g,\ell)|\leq 192n\ell^5\tau(n+\ell,g-1)\]
\end{lemma}
\begin{proof}

We prove the lemma by an injective operation. It is (sketchily) illustrated in~\cref{fig_fat_cut}.

\begin{figure}
\center
\includegraphics[scale=1]{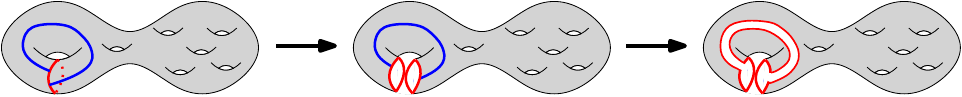}
\caption{Cutting along a \fe{}.}\label{fig_fat_cut}
\end{figure}

Let $t,c,p$ be an element of $\Tf(2n,g,\ell)$ where $t$ is the triangulation and $c,p$ is the \fe{}. As in the previous section, pick an arbitrary edge of $c$ to become the root, and the original root is now an marked oriented edge ($6n$ choices). Cut along $c$ (recall it is non separating) to obtain a triangulation of $\T_{|c|,|c|}(2n,g-1)$ with a marked path $p$. By definition of a \fe{}, $p$ must join both boundaries. So far, we have an injective operation because to go back one just needs to glue the two boundaries such that the roots coincide. Now, cut along $p$ to obtain a triangulation of $\T_{2\ell}(2n,g-1)$: one of the two roots becomes a marked edge on the boundary ($2\ell$ choices), and to remember how to close $p$ and go backwards, one can for instance\footnote{actually, we are marking more than we need to go back, but this won't matter in the asymptotic analysis.} mark $4$ vertices of the boundary, corresponding to both endpoints of $p$ before it was split ($(2\ell)^4$ choices).

Therefore
\[|\Tf(2n,g,\ell)|\leq 6n\times 2\ell\times (2\ell)^4\times\T_{2\ell}(2n,g-1)=  192 n\ell^5\T_{2\ell}(2n,g-1).\]
Invoking~\cref{lem_filling_boundaries} finishes the proof.
\end{proof}

We can now prove ~\cref{prop_fat}:
\begin{proof}[Proof of~\cref{prop_fat}]
Take $\ell\leq A\log n$, then by \cref{lem_fat,lem_asympto_genus_ratio}, we get
\[\Tf(2n,g_n,\ell)\leq n^{1+o(1)}\tau(n+\ell,g_n-1)\leq n^{1+o(1)}n^{-2+o(1)} \lambda(\theta)^\ell e^{o(\ell)}\leq n^{-1/2+o(1)}. \]
where in the last inequality we use the fact that $\lambda(\theta)^{A\log(n)}\leq \sqrt{n}$.
Summing this over all $\ell \leq A\log n$ yields the desired result.
\end{proof}

\subsection{Counting \te{s}}
In this section, we prove that $\bT$ does not contain small \te{s}, that is:
\begin{proposition}\label{prop_thin}
Let $A=\frac{1}{2\log\lambda(\theta)}$, we have
\[\sum_{\ell\leq A\log n} \frac{\Tt(2n,g_n,\ell)}{\tau(n,g_n)}\to 0\]
as $\nto$.
\end{proposition}

Let us start with an inequality that holds for all $n,g,\ell$.

\begin{lemma}\label{lem_thin}
We have 
\[|\Tt(2n,g,\ell)|\leq 6n(6(n+\ell))^3(2\ell)^7\tau(n+\ell,g-2)+ 6n(2\ell)^7\sum_{n_1+n_2=n+\ell}\sum_{g_1+g_2=g-1\atop g_1,g_2\geq 1}6n_1\tau(n_1,g_1)\tau(n_2,g_2)\]
\end{lemma}
\begin{proof}
We will write this proof more concisely, as it is of the same nature than the previous ones.
Is is once again an injective operation, illustrated in~\cref{fig_thin_cut}.

\begin{figure}
\centering
\includegraphics[scale=0.7]{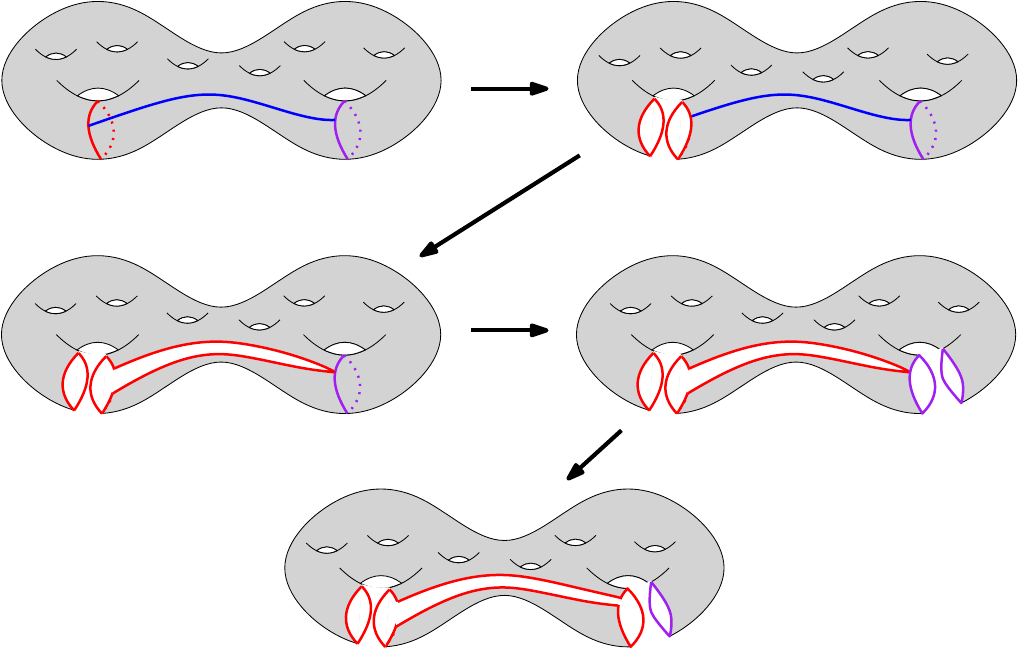}

\vspace{1cm}

\includegraphics[scale=0.7]{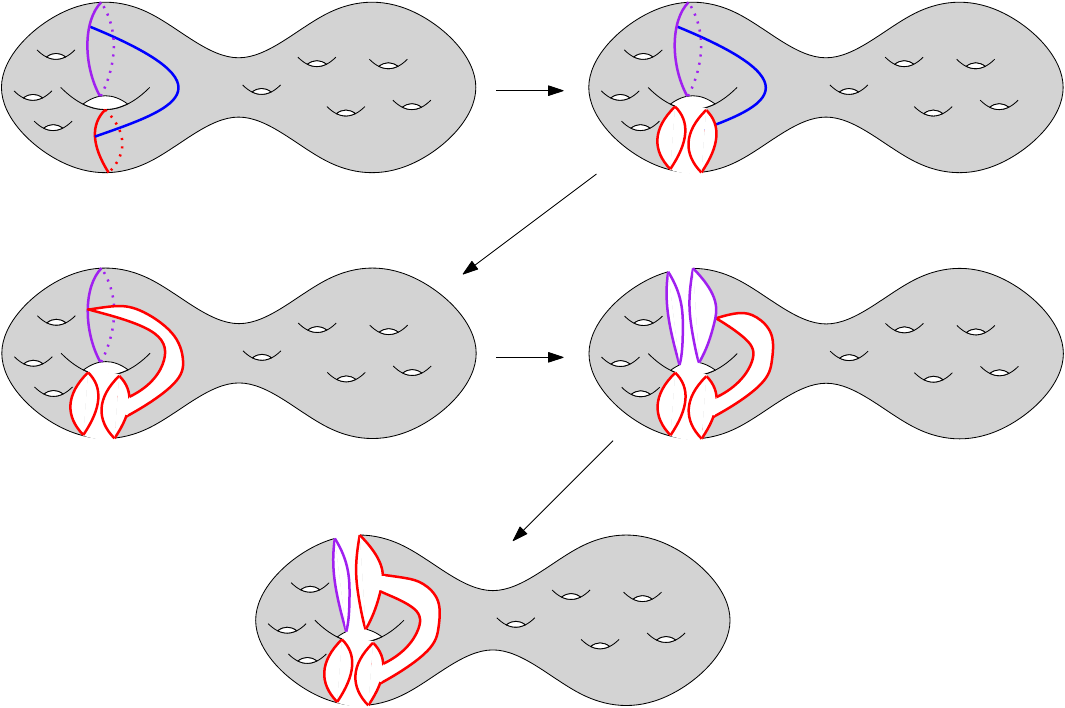}
\caption{Cutting along a \te{}, the two cases.}\label{fig_thin_cut}
\end{figure}
 Given $t$ and its \te{} $c_1,c_2,p$ in $\Tt(2n,g,\ell)$, reroot on $c_1$ and consider the previous root as a marked edge ($6n$ choices). Cut along $c_1$ to obtain a map of $\T_{|c_1|,|c_1|}(2n,g-1)$ with two boundaries, a marked path $p$ joining the first on to a marked cycle $c_2$. Slit open $p$ to obtain a map of $\T_{|c_1|+2|p|,|c_1|}(2n,g)$ with a marked cycle $c_2$ (as in the proof of~\cref{lem_fat}, mark an edge and $4$ vertices on the first boundary to go back, that's $\leq 2^5\ell^5$ choices).

Now, what has become of $c_2$ after all these surgeries ? By definition of a \te{}, it has remained a cycle, even if $p$ was empty. It cannot be contractible (otherwise it was bounding a disk from the start) homotopic to a boundary (otherwise it was homotopic to $c_1$ in the original map). However, it is possible that it has become separating although it started as a non-separating cycle.

There are two cases. If it stayed non-separating, cut along it to obtain a triangulation of genus $g-2$ with four boundaries of total size $2\ell$, two of which intersect at a vertex. Therefore they can be merged as a single boundary by opening at the intersection point (and to go back, one can for instance mark two vertices on that boundary, that is $\leq 4\ell^2$ choices). Using~\cref{lem_filling_boundaries}, this gives the first term in the RHS of the inequality.

If it has become separating, then it must separate the surface into two subsurfaces with both with genus (otherwise, it means that it separates a pair of pants from the rest, but in that case it was a separating cycle in the original map, a contradiction). If we cut along it, we obtain a pair of maps both with genus $1$ or greater, total genus $g-1$, two boundaries each in total, with total size $2\ell$. In one of the maps, the two boundaries touch at a point, hence they can be opened into a single boundary at the cost of remembering two vertices ($\leq 4\ell^2$ choices).
Once again, using~\cref{lem_filling_boundaries}, this gives the second term in the RHS of the inequality.

\end{proof}

The proof of~\cref{prop_thin} follows from applying the same arguments as in the proofs of \cref{prop_fat,prop_simple} to the first and second term in the RHS of the inequality of~\cref{lem_thin}. 

\section{Discussion and conjectures}\label{sec_conj}
\subsection{Enumeration}

By \cref{eq_asymptoo}, we know that $\tau(n,g)=n^{2g}\exp(nf(g/n)+o(n))$. To guess the genus ratio of \cref{thm_ratio}, one could be tempted to ignore the $e^{o(n)}$ and write
\[\frac{n^2\tau(n,g-1)}{\tau(n,g)}\approx \exp(nf(g-1/n)-nf(g/n))\approx \exp(-f'(g/n)).\]

And indeed, after some calculations that we omit here, it is possible to prove that $\psi(\theta)=\exp(-f'(\theta))$ where $\psi$ is the function of \cref{thm_ratio}. In other words, whatever is hidden in the $e^{o(n)}$ of \cref{eq_asymptoo} does not vary too wildly as $n$ and $g$ vary.

\subsection{Geometry}
As we mentioned in the introduction,  the mere existence of a simple \sncc{} in triangulations is still to be proven, this known as \emph{Barnette's conjecture}~\cite{MT01}[Chapter 5].
\begin{conjecture}\label{conj_Barnette}
Let $t$ be an arbitrary triangulation without loops or multiple edges, of genus at least $2$. Then $t$ contains a simple \sncc{}.
\end{conjecture}

So far, this conjecture still stands, despite several important advances (see e.g. \cite{BMR}). Proving that one cannot design an especially evil triangulation contradicting the conjecture seems out of reach for us, but since sampling objects at random often rules out many pathological phenomena, we believe it could apply to random high genus triangulations\footnote{although, since we allow loops and multiple edges, these are not included in Barnette's conjecture.}.

\begin{conjecture}\label{conj_simple}
For every $\theta\in(0,1/2)$, there exists a constant $K$ such that, whp, $\bT$ contains a simple, non-contractible, separating cycle of length less than $K\log(n)$.
\end{conjecture}

As for many other geometric quantities, we expect the value of the separating systole to concentrate asymptotically.

\begin{conjecture}\label{conj_constant}
Let $\sep^*$ be the simple separating systole of $\bT$.
For every $\theta\in(0,1/2)$, there exists two constants $0<S(\theta)<S^*(\theta)$ such that
\[\frac{\sep}{\log(n)}\to S(\theta)\quad \text{and}\quad \frac{\sep^*}{\log(n)}\to S^*(\theta)\]
in probability as $\nto$.
\end{conjecture}

Note that we expect the simple separating systole to be significantly larger than the separating systole.

We finish this section with a deterministic conjecture regarding the structure of the shortest \sncc{} in a triangulation.

\begin{conjecture}\label{conj_decomp}
Let $t$ be a triangulation and $C$ a \sncc{} of $t$ of shortest length. Then one cannot write $C=\gamma^2C'$, where $\gamma$ is a simple cycle.
\end{conjecture}

Note that this conjecture is very easy to prove unless $C$ is non simple and $\gamma$ is non separating. Even in that case, intuitively, starting by doing two loops around $\gamma$ seems wasteful if one seeks to minimize the length, and perhaps from any \sncc{} of the form $\gamma^2C'$ one can find another, shorter \sncc{} $\gamma C''$ where $C''$ is "$C'$ with an occurrence of $\gamma\inv$ removed".

\bibliographystyle{abbrv}
\bibliography{bibli}

\end{document}